\documentclass[10pt]{amsart}
\usepackage[utf8]{inputenc}
\usepackage[T1]{fontenc}
\usepackage{libertine}
\usepackage{bm}

\usepackage{amssymb}
\usepackage{amscd}
\usepackage{amsfonts}
\usepackage{setspace}
\usepackage{version}

\usepackage{graphicx}

\newtheorem{theorem}{Theorem}[section]
\newtheorem{lemma}[theorem]{Lemma}
\newtheorem{proposition}[theorem]{Proposition}

\theoremstyle{definition}

\newtheorem{question}{Question}

\renewcommand{\leq}{\leqslant}
\renewcommand{\geq}{\geqslant}

\newcommand\SO{\operatorname{SO}}

\newcommand\tr{\operatorname{tr}}

\def\R{\mathbf{R}}

\def\Z{\mathbf{Z}}

\def\G{\mathbf{G}}
\def\d{\mathbf{d}}
\def\calN{\mathcal{N}}
\newcommand\counting{\operatorname{triv}}
\def\bmu{{\bm \mu}}
\def\H{\operatorname{H}}
\def\eps{\varepsilon}

\numberwithin{equation}{section}

\begin{document}

\title[A conjecture of Gowers and Long]{On a conjecture of Gowers and Long}


\author{Ben Green}
\address{Mathematical Institute\\
Radcliffe Observatory Quarter\\
Woodstock Road\\
Oxford OX2 6GG\\
England }
\email{ben.green@maths.ox.ac.uk}

\subjclass[2000]{Primary }

\thanks{The author is supported by a Simons Investigator Grant.}

\begin{abstract}
We show that rounding to a $\delta$-net in $\SO(3)$ is not close to a group operation, thus confirming a conjecture of Gowers and Long.
\end{abstract}
\maketitle

\section{Introduction}

In a very interesting recent preprint \cite{gowers-long}, Gowers and Long considered somewhat associative binary operations. In their paper they describe a very natural example of a somewhat associative operation, and conjecture (see \cite[Conjecture 1.6]{gowers-long}) that it does not resemble a genuine group operation. Our aim in this note is to prove their conjecture. 

Let us begin by describing their example. Take $\SO(3)$ with the group operation denoted by juxtaposition and with, for definiteness, the (bi-invariant) metric $d : \SO(3) \times \SO(3) \rightarrow [0,2^{3/2}]$ given by $d(g,h) := \Vert g  - h \Vert$, where $\Vert g \Vert$ is the Frobenius (or Hilbert-Schmidt) norm $\Vert M \Vert := \sqrt{\tr(M^T M)}$. 

Throughout the paper we will take $\delta > 0$ be a small parameter, and let $X$ be a maximal $\delta$-separated subset of $\SO(3)$. We have $|X| \sim \delta^{-3}$. Define a binary operation $\circ : X \times X \rightarrow X$ by defining $x \circ y$ to be the nearest point of $X$ to $xy$ (breaking ties arbitrarily). Since $X$ is assumed to be maximal $\delta$-separated, we always have
\begin{equation}\label{proximal} d(x \circ y, xy) \leq \delta. \end{equation}

\textbf{Claim} (Gowers--Long). For a positive proportion of triples $(x_1, x_2, x_3) \in X^3$ we have the associativity relation $(x_1 \circ x_2) \circ x_3 = x_1 \circ (x_2 \circ x_3)$.\vspace{8pt}

Gowers and Long note that it seems very unlikely that any substantial portion of the multiplication table of $\circ$ can be embedded in a group operation, and make a precise conjecture, \cite[Conjecture 1.6]{gowers-long}, to this effect. The following result establishes their conjecture.

\begin{theorem}\label{mainthm}
Suppose that $\iota : X \rightarrow G$ is an injective map into a group $G$, with group operation $\cdot$. Then the number of pairs $(x_1,x_2) \in X \times X$ with $\iota(x_1) \cdot \iota(x_2) = \iota(x_1 \circ x_2)$ is at most $\eps |X|^2$, where $\eps \rightarrow 0$ as $\delta \rightarrow 0$.
\end{theorem}  

We remark that we do not obtain any effective information on the speed at which $\eps \rightarrow 0$. This is because we rely on the structure theory of approximate groups \cite{bgt}, which uses ultrafilter arguments. 

\emph{Notation.} Our notation is fairly standard. Occasionally we will write, for instance, $o_{K; \delta \rightarrow 0}(1)$, which means some quantity tending to zero as $\delta \rightarrow 0$, but the rate at which this happens may depend on the parameter $K$. We write $X \gg Y$ to mean $X \geq cY$ for some absolute $c > 0$, and $X \sim Y$ means $Y \ll X \ll Y$.

Recall that $\Vert \cdot \Vert$ denotes the Hilbert-Schmidt norm and that we use this to define a distance on $\SO(3)$ by $d(g,h) := \Vert g - h\Vert$, where $\SO(3)$ is embedded in the space of $3$-by-$3$ matrices by fixing an orthonormal basis for $\R^3$. We write $|g| = d(g,1)$ for the distance to the identity (which we will always denote by $1$, the underlying group hopefully being clear from context). It may be computed that $|g| = 2^{3/2} |\sin(\theta/2)|$, where $\theta$ is the angle of the rotation $g$. Recall that the Hilbert-Schmidt norm is submultiplicative (that is, $\Vert a b \Vert \leq \Vert a \Vert \Vert b \Vert$ for all real 3-by-3 matrices). Additionally, using the conjugation invariance of trace and the fact that $g^T = g^{-1}$ for $g \in \SO(3)$, we have the $\SO(3)$-invariance $\Vert a \Vert = \Vert a g \Vert = \Vert g a \Vert$ for all 3-by-3 matrices $a$ and all $g \in \SO(3)$. In particular, $g N_{\delta}(1) = N_{\delta}(1) g$ for all $g \in \SO(3)$, where $N_{\delta}(1)$ denotes the $\delta$-neighbourhood of $1$.

\emph{Acknowledgement.} I would like to thank Emmanuel Breuillard for helpful conversations.

\section{An initial reduction}

We can fairly quickly reduce the task of proving Theorem \ref{mainthm} to that of establishing the following proposition which, since it does not involve the awkward $\circ$, is of a more conventional type. 

\begin{proposition}\label{mainprop}
Let $\eps, \delta > 0$. Let $(G, \cdot)$ be a group, and let $A$ be a finite subset of $G$ of size $n$. Let $f : A \rightarrow \SO(3)$ be a map with $\delta$-separated image, and with the property that there are at least $\eps n^3$ quadruples $(a_1, a_2, a_3, a_4) \in A^4$ with
$a_1 \cdot a_2 = a_3 \cdot a_4$ and $d(f(a_1) f(a_2) , f(a_3) f(a_4)) \leq \delta$. Then $n = o_{\eps; \delta \rightarrow 0}(\delta^{-3})$.
\end{proposition}

\begin{proof}[Proof of Theorem \ref{mainthm}, assuming Proposition \ref{mainprop}]
Let $X$, $|X| = n$, be a maximal $\delta$-net in $\SO(3)$, suppose that $\iota : X \rightarrow G$ is an injection from $X$ into a group $(G, \cdot)$ and that there are $\eps n^2$ pairs $(x_1, x_2) \in X \times X$ with $\iota(x_1) \cdot \iota(x_2) = \iota(x_1 \circ x_2)$.  Take $A := \iota(X) \subset G$, and let $f : A \rightarrow \SO(3)$ be the inverse of $\iota$. Let $\Omega \subset A \times A$ be the set of all pairs $(a_1, a_2)$, $a_1 = \iota(x_1)$, $a_2 = \iota(x_2)$, such that $a_1 \cdot a_2 = \iota(x_1 \circ x_2)$. Thus $|\Omega| \geq \eps n^2$, and if $(a_1, a_2) \in \Omega$ then $a_1 \cdot a_2 \in A$. For $a \in A$, let $r(a)$ denote the number of pairs $(a_1, a_2) \in \Omega$ such that $a_1 \cdot a_2 = a$. Thus $\sum_{a \in A} r(a) \geq \eps n^2$ and so, by Cauchy-Schwarz, 
\[ \sum_{a \in A} r(a)^2 \geq \frac{1}{|A|}\big( \sum_{a \in A} r(a)\big)^2 \geq \eps^2 n^3.\]
The sum on the left is counting the number of quadruples $(a_1,a_2,a_3,a_4)$ with $a_1 \cdot a_2 = a_3 \cdot a_4$ and $(a_1, a_2), (a_3, a_4) \in \Omega$. From the definition of $\Omega$ we have, for any such quadruple, 
\[ f(a_1) \circ f(a_2) = f(a_1 \cdot a_2) = f(a_3 \cdot a_4) = f(a_3) \circ f(a_4).\]
By \eqref{proximal} we also have
\[ d(f(a_1) \circ f(a_2), f(a_1) f(a_2)), d(f(a_3) \circ f(a_4), f(a_3) f(a_4)) \leq \delta.\]
It follows from the triangle inequality that 
\[ d(f(a_1) f(a_2), f(a_3) f(a_4)) \leq 2 \delta.\]
Applying Proposition \ref{mainprop} (with $\delta$ replaced by $2\delta$ and $\eps$ by $\eps^2$), we see that $|X| = n = o_{\eps; \delta \rightarrow 0}(\delta^{-3})$, a contradiction if $\delta$ is small enough as a function of $\eps$.
\end{proof}

\section{Outline of the rest of the argument}\label{outline-sec}

In this section we outline the rest of the argument. Recall that a \emph{$K$-approximate group} is a subset $B$ of some ambient group which is symmetric (that is, it contains the identity $1$, and $B^{-1} = B$) and such that $B^2$ is covered by $K$ left- (or equivalently right-) translates of $B$. See \cite{tao} for further discussion and background. Note that approxmate groups need not be finite.

In the next section, we show that the existence of a map $f$ as in Proposition \ref{mainprop} would imply the existence of an approximate homomorphism from a finite approximate group to $\SO(3)$ with a ``thick'' image.  In discussing approximate homomorphisms $\phi$ it is natural to introduce the notion of \emph{cocyle}, defining $\partial \phi(x,y) := \phi(y)^{-1} \phi(x)^{-1}\phi(xy)$.

\begin{proposition}\label{sec4-main}
Let $\eps, \delta > 0$. Let $(G, \cdot)$ be a group, and let $A$ be a finite subset of $G$ of size $n \sim \delta^{-3}$. Let $f : A \rightarrow \SO(3)$ be a map with $\delta$-separated image, and with the property that there are at least $\eps n^3$ quadruples $(a_1, a_2, a_3, a_4) \in A^4$ with $a_1 \cdot a_2 = a_3 \cdot a_4$ and $d(f(a_1) f(a_2) , f(a_3) f(a_4)) \leq \delta$.
Then there is a finite $K$-approximate group $B \subset G$, $K \ll_{\eps} 1$, and a map $\phi : B^{12} \rightarrow \SO(3)$ with the following two properties. First, $\phi$ has thick image in the sense that $\mu(N_{\delta}(\phi(B^4))) \sim_{\eps} 1$. Secondly, $\phi$ is an approximate homomorphism in the sense that there is a set $S \subset \SO(3)$, $|S| \ll_{\eps} 1$, such that whenever $x,y, xy \in B^{12}$ we have $d(\partial\phi(x,y), S) \leq \delta$.
\end{proposition}

\emph{Remarks.} Here $N_{\delta}$ means the $\delta$-neighbourhood (in the metric $d$) and $\mu$ is the normalised Haar measure on $\SO(3)$. As we shall see, no particular properties of $\SO(3)$ are used in the proof, beyond the existence of $d$ and $\mu$ and their basic properties. Note that the ``approximateness'' of $\phi$, whilst of two different types (the error set $S$ and the parameter $\delta$) is all in the \emph{range}, whereas $f$ is approximate in the domain, in that the weak homomorphism property only holds some of the time. This idea of moving the ambiguity from the domain to the range follows a line of argument pioneered by Gowers in his seminal works \cite{gowers-1, gowers-2} (based also on work of Ruzsa). Proposition \ref{sec4-main} is a consequence of the metric entropy version of the noncommutative Balog--Szemer\'edi-Gowers theorem of Tao \cite{tao}.
 
 Proposition \ref{mainprop}, and hence Theorem \ref{mainthm}, follows immediately from Proposition \ref{sec4-main} and the next result, which says that the two properties of $\phi$ in the conclusion of Proposition \ref{sec4-main} are incompatible: an approximate homomorphism from a finite approximate group to $\SO(3)$ has a thin image.

 \begin{proposition}
\label{sec-5-main} Let $B$ be a finite $K$-approximate group, let $S \subset \SO(3)$ be a set of size at most $K$, and suppose that $\phi : B^{12} \rightarrow \SO(3)$ satisfies $d(\partial\phi(x,y) , S) \leq \delta$ whenever $x,y, xy \in B^{12}$. Then $\mu(N_{\delta}(\phi(B^4))) = o_{K; \delta \rightarrow 0}(1)$.
 \end{proposition}

The proof of this uses quite different techniques and we appeal to some fairly specific features of $\SO(3)$, though it would probably be possible to adapt the proof so as to work with $\SO(3)$ replaced by, for example, any simple Lie group. We divide into two cases, according to whether or not the cocycle $\partial \phi$ takes values far from the identity. Recall that, for $g \in \SO(3)$, $|g|$ means $d(g, 1)$.

\emph{Case 1.} There exist $x,y \in B^4$ with $|\partial\phi(x,y)| > \sqrt{\delta}$. Then a fairly direct argument shows that $\phi(B^4)$ must lie in a union of $O(1)$ translates of an ``almost centraliser'' of $\partial\phi(x,y)$, and a further argument shows this has small measure. 

\emph{Case 2.} We have $|\partial\phi(x,y)| < \sqrt{\delta}$ for all $x,y \in B^4$, that is to say $\phi$ satisfies $\phi(xy) \approx \phi(x) \phi(y)$ up to an error of $\sqrt{\delta}$ in the range. If $B = B^4$ were actually a finite group, a result of Kazhdan \cite{kazhdan} then implies that we can correct $\phi$ by $O(\sqrt{\delta})$ to get a genuine homomorphism $\tilde\phi : B \rightarrow \SO(3)$. In particular, $\phi(B^4)$ lies within $O(\sqrt{\delta})$ of a finite subgroup of $\SO(3)$. However, these are all cyclic, dihedral or contained in $S_5$ and hence $\phi(B^4)$ is ``thin'' in the sense discussed above. It is tempting to try and minic the arguments of \cite{kazhdan} when $B$ is merely an approximate group, but this does not work in any obvious way. Rather we use the classification of approximate groups due to Breuillard, Tao and the author \cite{bgt} and then invoke Kazhdan's result as a black box, ending up showing that $\phi(B^4)$ can be partitioned into $O(1)$ pieces, each of which almost satisfies a nontrivial word equation, which then implies that $\phi(B^4)$ is thin. In particular we do not prove that $\phi$ can be corrected to a homomorphism $\tilde \phi$; it would be interesting to explore this direction. For further remarks see Section \ref{sec7}.

\section{An application of Tao's metric entropy BSG theorem}

In this section we establish Proposition \ref{sec4-main} . Let $\G := G \times \SO(3)$, and let $\d : \G \times \G \rightarrow \G$ be the product of the discrete (extended) metric $d_{\counting}$ on $G$, where the distance between distinct points is $\infty$, and the metric $d$ on $\SO(3)$. Let $\bmu = \mu_{\counting} \times \mu$, where $\mu_{\counting}$ is the counting measure on $G$ (that is, the measure of any finite set $A \subset G$ is simply $|A|$) and $\mu$ is normalised Haar measure on $\SO(3)$. The group $\G$, endowed with the measure $\bmu$ and the (extended) metric $\d$, is \emph{locally reasonable} in the sense of Tao \cite[Definition 6.3]{tao}\footnote{Although Tao does not explicitly allow extended metrics, this creates no problems in his arguments, and is necessary to ensure the doubling property $\mu_{\counting}(B(2r)) \sim \mu_{\counting}(B(r))$ for balls $B()$ in the discrete (extended) metric. In fact if one applies his results to the discrete (extended) metric, one recovers the standard finitary theory of noncommutative sumset estimates.}.

To state the result from \cite{tao} that we will need, we recall the definition of covering numbers used in that paper: if $X$ is a subset of a metric space, $\calN_{\eta}(X)$ is the least number of balls of radius $\eta$ necessary to cover $X$. We also define the $\eta$-approximate multiplicative energy $E_{\eta}(X,X)$ of a set to be $\calN_{\eta}(Q_{\eta}(X,X))$, where
\[ Q_{\eta}(X,X) := \{ (x_1, x_2, x_3,x_4) \in X^4 : \d(x_1x_2, x_3x_4) \leq \eta\}.\] The metric entropy here is with respect to the product metric on $X^4$.

The following is the implication (i) $\Rightarrow$ (iv) of \cite[Theorem 6.10]{tao}, specialised to our setting.

\begin{proposition}[Tao]\label{tao-conseq}
Suppose that $E_{\eta}(X,X) \geq \frac{1}{K} \calN_{\eta}(X)^3$. Then there exists a $K^{O(1)}$-approximate subgroup $H \subset \G$ and an element $g \in \G$ such that \emph{(1)} $\calN_{\eta}(H) \sim K^{O(1)}\calN_{\eta}(X)$ and \emph{(2)} $\calN_{\eta}(X \cap g H)\sim K^{O(1)}\calN_{\eta}(X)$.
\end{proposition}

\begin{proof}[Proof of Proposition \ref{sec4-main}]
Let us first recall the hypotheses under which we are operating, which are those of Proposition \ref{sec4-main}, namely that $f : A \rightarrow \SO(3)$ is a map with the property that there are at least $\eps n^3$ quadruples $(a_1,a_2,a_3,a_4)$ with $a_1 \cdot a_2 = a_3 \cdot a_4$ and $d(f(a_1) f(a_2) , f(a_3) f(a_4)) \leq \delta$.

Take $X = \{(a, f(a)) : a \in A\} \subset \G$ to be the graph of $f$. Let $\pi : \G \rightarrow G$ be projection. Then, since $\pi$ is injective on $X$ and the metric on $G$ is discrete,
\[ \calN_{\delta}(X) = |X| = |A| = n.\]

By assumption, $|Q_{\delta}(X, X)| \geq \eps n^3$.  Since $\pi^{\otimes 4} : \G^4 \rightarrow G^4$ is injective on $Q_{\delta}(X, X)$, we have $E_{\delta}(X,X) = |Q_{\delta}(X,X)| \geq \eps n^3$.
Therefore the hypothesis of Proposition \ref{tao-conseq} is satisfied with $K = \eps^{-1}$. For the rest of the proof of Proposition \ref{sec4-main}, all instances of the $\gg$, $\ll$ and $O()$ notations may depend on $\eps$ but this will not be explicitly indicated. Applying Proposition \ref{tao-conseq}, we obtain an $O(1)$-approximate subgroup $H \subset \G$ and an element $g = (x,y) \in \G$ satisfying \begin{equation}\label{p41-conseq} \calN_{\delta}(H) \sim \calN_{\delta}(X \cap gH) \sim \calN_{\delta}(X) = n.\end{equation} 
Let $B_0 := \pi(g^{-1} X \cap H)$. Using the fact that $\pi$ is injective on $X$ and that $G$ is discrete, we have
\begin{equation}\label{b0-size} |B_0| = |\pi(g^{-1} X \cap H)| = |X \cap gH| \sim n.\end{equation}
Using the fact that the metric on $G$ is discrete once more, we also have $|\pi(H)| \leq \calN_{\delta}(H)$, and hence from \eqref{p41-conseq}, \eqref{b0-size} it follows that 
\begin{equation}\label{H-upper} |\pi(H)| \sim n.\end{equation}
(Note that $H$ itself may well be infinite). Let $\phi : B_0 \rightarrow \SO(3)$ be such that $(b, \phi(b)) \in g^{-1} X \cap H$ for all $b \in B_0$. Then $\phi(b)$ takes values in $y^{-1} f(A)$ and hence (since the metric on $\SO(3)$ is bi-invariant) is $\delta$-separated.

Now the property of being an approximate group is preserved under $\pi$. Thus $\pi(H)$ is an $O(1)$-approximate group and in particular $|\pi(H)^3| \ll |\pi(H)|$. From \eqref{b0-size}, \eqref{H-upper} it follows that $|B_0^3| \ll |B_0|$. Therefore by \cite[Corollary 3.11]{tao} we see that $B := (B_0 \cup \{1\} \cup B_0^{-1})^3$ is an $O(1)$-approximate group.

Extend $\phi$ to a map from $B^{12}$ to $\SO(3)$ as follows: for each $x \in B^{12} \setminus B_0$, write $x = b_1^{\eps_1} \cdots b_{36}^{\eps_{36}}$ with $\eps_1,\dots, \eps_{36} \in \{-1,0,1\}$ and $b_1,\dots, b_{36} \in B_0$. If there is more than one such representation of a given $x$, choose one arbitrarily. Now define
\[ \phi(x) := \phi(b_1)^{\eps_1} \cdots \phi(b_{36})^{\eps_{36}}.\]
Note that $\phi(B^4)$ contains $\phi(B_0)$ which, as observed above, is a collection of $\sim n$ $\delta$-separated points. Since $n \sim \delta^{-3}$ (and the volume of a $(\delta/2)$-neighbourhood in $\SO(3)$ is $\sim \delta^3$), it follows that $\mu(N_{\delta}(\phi(B^4))) \sim 1$.

To conclude the proof, we must show that the cocycle $\partial\phi(x,y)$ takes values $\delta$-close to some small set $S$, whenever $x,y, xy \in B^{12}$. Since $\{ (b, \phi(b)) : b \in B_0\} \subset H$, we see that if $x, y, xy \in B^{12}$ then $\partial\phi(x,y) = \phi(y)^{-1} \phi(x)^{-1}\phi(xy)$ lies in the fibre $F := \pi^{-1}(1) \cap H^{36}$. To conclude the proof of Proposition \ref{sec4-main}, it is therefore enough to prove that 
\begin{equation}\label{F-bd} \calN_{\delta}(F) \sim 1.\end{equation}

To prove \eqref{F-bd}, observe first that, since $H$ is a $K$-approximate group for some $K = O(1)$, $H^{37}$ is covered by $K^{36}$ translates of $H$, and so (by \eqref{p41-conseq} and the bi-invariance of the metric) we have 
\begin{equation}\label{n37} N_{\delta}(H^{37}) \sim n.\end{equation} However, $H^{37}$ contains a translate of $F$ above every point of $\pi(H)$, and thus
\[ |\pi(H)| \calN_{\delta}(F) \leq N_{\delta}(H^{37}).\]
The desired estimate \eqref{F-bd} follows immediately from this, \eqref{H-upper} and \eqref{n37}.
\end{proof}

\section{Case 1: a large element in the error set}

We turn now to the proof of Proposition \ref{sec-5-main}. The reader may wish to recall the outline given in Section \ref{outline-sec}. In this section we look at the first case discussed there, in which there are $x,y \in B^4$ such that $|\partial\phi(x,y)| > \sqrt{\delta}$. Before giving the main argument, let us record a lemma concerning almost commuting rotations. This must surely exist in the literature but I could not locate a reference. Here, and in what follows, we define the conjugate $a^g$ to be $g^{-1} a g$ and the commutator $[a,g]$ to be $a^{-1} g^{-1} a g$.

\begin{lemma}\label{lem5.1}
Let $a, g \in \SO(3)$, $a \neq 1$.  Then $d(g, C(a)) \ll \frac{|[a,g]|}{|a|}$, where $C(a)$ denotes the centraliser of $a$.
\end{lemma}
\begin{proof}
It is easy to check that if the statement is true for $a$, then it is true for any conjugate of $a$, and thus we may assume that \[ a = r(\theta) := \left(\begin{smallmatrix} \cos \theta & \sin \theta & 0 \\ -\sin \theta & \cos \theta & 0 \\ 0 & 0 & 1 \end{smallmatrix}\right).\] 
By the existence of Euler angles, every $g$ can be written as $g = r(\beta_1)  r'(\alpha) r(\beta_2)$, where 
\[r'(\alpha) := \left(\begin{smallmatrix} 1 & 0 & 0 \\ 0 & \cos \alpha & \sin\alpha \\ 0 & -\sin \alpha & \cos \alpha \end{smallmatrix}\right).\]
Since $r(\beta_1), r(\beta_2)$ commute with $a$ we have $|[a,g]| = |[a, r'(\alpha)]|$. A computation gives
\[  |[a, r'(\alpha)]| = d( r(\theta) r'(\alpha), r'(\alpha) r(\theta)) =  2^{5/2} xy (1 - x^2 y^2)^{1/2},\]
where $x := |\sin(\theta/2)|$, $y = |\sin(\alpha/2)|$.
Thus, writing $\eta := |[a,g]|$, we see that either (i)  $xy \ll \eta$, or (ii) $y \geq 1 - O(\eta^2)$. Note also that $x = 2^{-3/2}|a|$. 

In case (i), $|\tilde r(\alpha)| = 2^{3/2} |\sin(\alpha/2)| = 2^{3/2} y \ll \eta/|a|$. Therefore
\[ d(g, C(a)) \leq d(g, r(\beta_1 + \beta_2))  = \Vert r(\beta_1) (r'(\alpha) - 1) r(\beta_2) \Vert \leq 8 | r'(\alpha)| \ll \eta/|a|,\] where in the penultimate step we used the submultiplicativity of $\Vert \cdot \Vert$. This concludes the proof in this case.

In case (ii), $d( r'(\alpha), r'(\pi)) = 2^{3/2} |\cos(\alpha/2)| \ll \eta$. Noting that $r'(\pi)$ commutes with $a$, we have
\[ d(g, C(a)) \leq d(g, r(\beta_1) r'(\pi) r(\beta_2)) = \Vert r(\beta_1)(r'(\alpha) - r'(\pi)) r(\beta_2) \Vert \ll \eta.\] This completes the proof of the lemma.
\end{proof}

Now we return to the proof of Proposition \ref{sec-5-main} (first case). Let $z \in B^4$ be arbitrary. By writing $\phi(xyz)$ in two different ways one easily obtains the cocycle equation
\begin{equation}\label{cocycle-eq} a^{\phi(z)}  = \partial\phi(y,z) \partial\phi(x,yz)\partial\phi(xy,z)^{-1} ,\end{equation} where $a := \partial\phi(x,y)$. Since $x,y,z \in B^4$, all the pairwise products as well as the triple product $xyz$ lie in $B^{12}$, and hence by the hypotheses of Proposition \ref{sec-5-main} the three cocycles $\partial\phi(y,z)$, $\partial\phi(x,yz)$, $\partial\phi(xy,z)$ lie in the $\delta$-neighbourhood of $S$. It follows from \eqref{cocycle-eq} that, for all $z \in B^4$, $d(a^{\phi(z)}, SSS^{-1}) \leq 3 \delta$. Consequently, we may find a set $z_1,\dots, z_k$, $k \leq |SSS^{-1}| \leq K^3$, of elements of $B^4$ such that for every $z \in B^4$ there is some $i$ such that $d(a^{\phi(z)}, a^{\phi(z_i)}) \leq 6 \delta$. Equivalently, $|[a, \phi(z)\phi(z_i)^{-1}]| \leq 6 \delta$.  By Lemma \ref{lem5.1} we see that for every $z \in B^4$ there is some $i$ such that $d(\phi(z) \phi(z_i)^{-1}, C(a)) \ll \sqrt{\delta}$. Thus $\phi(B^4)$ is contained in the $(\sqrt{\delta})$-neighbourhood of at most $k$ translates of $C(a)$, a set whose measure tends to $0$ as $\delta \rightarrow 0$, uniformly in $a \neq 1$. This concludes the proof in the first case.

\section{Case 2: Almost homomorphisms}\label{sec6}

We now turn to the second case of Proposition \ref{sec-5-main}. This is the case in which $|\partial\phi(x,y)| \leq \sqrt{\delta}$ whenever $x,y,xy \in B^{12}$, and we wish to conclude that $\mu(N_{\delta}(\phi(B^4))) = o_{K; \delta \rightarrow 0}(1)$. For notational convenience, redefine $\sqrt{\delta}$ to $\delta$, thus we have $|\partial\phi(x,y)| \leq \delta$, or equivalently $\phi$ satisfies the almost-homomorphism property
\begin{equation}\label{cocycle-small} d(\phi(xy), \phi(x)\phi(y)) \leq \delta\end{equation} whenever $x,y,xy \in B^{12}$. We wish to conclude that $\mu(N_{\delta}(\phi(B^4))) = o_{K; \delta \rightarrow 0}(1)$.

We will repeatedly use the fact, easily established using \eqref{cocycle-small} and induction, that if $Q$ is a symmetric set with $Q^m \subset B^{12}$ then
\begin{equation}\label{approx-word} d( \phi(w(x,y)), w(\phi(x), \phi(y))) \leq m \delta\end{equation} for all $x,y \in Q$, where $w$ is any word of length at most $m$ in the variables $x,y$. Of particular interest to us will be the commutator words $w_1(a,b) := [a,b]$, $w_{i+1}(a,b) := [a, w_i(a,b)]$. The length of $w_s$ is $\ell(w_s) = 3 \cdot 2^s - 2$.

We will also use the following result of Breuillard, Tao and the author \cite{bgt}.

\begin{theorem}\label{bgt-forus}
Suppose that $B$ is a $K$-approximate group. Then there is some $s = O_K(1)$, a symmetric set $Q$ of size $\gg_{K} |B|$ and a finite group $H \subset B^4$ such that
\begin{enumerate}
\item $Q^m \subset B^4$, where $m = 10 \ell(w_s)$;
\item If $x,y \in Q^4$ then $w_s(x,y) \in H$.
\end{enumerate}
\end{theorem}
\begin{proof}
\cite[Theorem 2.10]{bgt} states that $B^4$ contains an $O_K(1)$-proper coset nilprogression $P = P_H(u_1,\dots, u_r; N_1,\dots, N_r)$ with rank $r = O_K(1)$, step $s = O_K(1)$, and with $|P| \gg_K |B|$. We refer the reader to \cite[Section 2]{bgt} for the definitions required here, though the reader can fairly happily treat these concepts as black boxes for the purpose of this discussion. In particular, since $P$ contains $H$, so does $B^4$. Set $m = 10 \ell(w_s)$, thus $4 \leq m \ll_K 1$, and let $Q := P_H(u_1,\dots, u_r; \frac{1}{m} N_1, \dots, \frac{1}{m} N_r)$. It follows from the definitions in \cite[Section 2]{bgt} that $Q$ is symmetric and $Q^m \subset P$, and it follows from \cite[Lemma C.1]{bgt} that $|Q| \gg_{K,m} |P|$. Finally, if $x,y \in Q^4$ then certainly $x,y \in P$, and so from the fact that $P/H$ is $s$-step nilpotent we see that indeed $w_s(x,y) \in H$.
\end{proof}

From now on we drop explicit mention of $K$; all bounds can (and will) depend on $K$. It follows from the (nonabelian) Ruzsa covering lemma \cite[Lemma 3.6]{tao} that $B^4$ is a union of $O(1)$ translates $Q^2 g_i$, where $g_i \in B^{4}$. Evidently it suffices to show that $\mu(N_{\delta}(\phi(Q^2 g_i))) = o_{\delta \rightarrow 0}(1)$ for each $i$. Fix some $i$ and set $g := g_i$.

Suppose that $x_1, x_2, x_3, x_4 \in Q^2 g$. Then $x_1 x_2^{-1}, x_3 x_4^{-1} \in Q^4$, and so by Theorem \ref{bgt-forus} (2), $w_s(x_1x_2^{-1}, x_3 x_4^{-1}) \in H$. It follows from \eqref{approx-word} (and Theorem \ref{bgt-forus} (1)) that
\begin{equation}\label{word-1-small} d( w_s(\phi(x_1 x_2^{-1}), \phi(x_3 x_4^{-1})), \phi(H) ) =  O(\delta).\end{equation} 
Since $Q^4 \subset B^4$ and $g \in B^{4}$, we have $x_1, x_2^{-1}, x_3, x_4^{-1}, x_1 x_2^{-1}, x_3 x_4^{-1} \in B^{8}$, and so by \eqref{cocycle-small} we have
\[ d(\phi(x_1x_2^{-1}), \phi(x_1)\phi(x_2)^{-1}), d(\phi(x_3 x_4^{-1}), \phi(x_3) \phi(x_4)^{-1}) \leq \delta.\]
Using this many times in \eqref{word-1-small} (and the fact that the $\delta$-neighbourhood of $1$ is normalised by $\SO(3)$, to move all the errors of $\delta$ to the right) we obtain
    \begin{equation}\label{eq7} d(\tilde w_s (\phi(x_1), \phi(x_2), \phi(x_3), \phi(x_4)), \phi(H) ) = O(\delta),\end{equation} where $\tilde w_s(t_1, t_2, t_3, t_4) := w_s(t_1t_2^{-1}, t_3 t_4^{-1})$. Now since $H \subset B^4$ we see that $\phi$ is defined on all of $H$, and of course it still satisfies the approximate homomorphism condition \eqref{cocycle-small}. It is known that under these conditions there is a genuine homomorphism $\tilde\phi : H \rightarrow \SO(3)$ such that $d(\phi(h), \tilde\phi(h)) = O(\delta)$ for all $h \in H$. For the proof\footnote{Kazhdan acknowledges that in the compact case the result was obtained earlier by Grove, Karcher and Ruh \cite{gkr}, and in fact similar ideas go back to Turing \cite{turing}.}, see Kazhdan \cite{kazhdan}.
Thus, writing $\Sigma \subset \SO(3)$ for the subgroup $\tilde\phi(H)$, it follows that 
\begin{equation}\label{eq4} d(\tilde w_s (\phi(x_1), \phi(x_2), \phi(x_3), \phi(x_4)), \Sigma) = O(\delta)\end{equation} for all $x_1, x_2, x_3, x_4 \in Q^2 g$.
However, it is well-known\footnote{One could get away with weaker results here, such as Jordan's theorem.} that all finite subgroups of $\SO(3)$ are either cyclic, dihedral, or isomorphic to a subgroup of $S_5$. In particular there is some fixed universal word $w_*$ (for instance, $w_*(a,b) = [[a,b]^b, [a,b]]^{60}$), which is trivial on $\Sigma \times \Sigma$. Since $w_*$ is Lipschitz\footnote{Any word map is Lipschitz with Lipschitz constant the length of the word, since if $d(t_i, t'_i) \leq \delta$ then $d(t_1 \cdots t_m, t'_1 \cdots t'_m) \leq m \delta$, by an easy induction.}, it follows from \eqref{eq4} that for all $y_1, \dots, y_8 \in \phi(Q^2 g)$ we have
\[ |w(y_1,\dots, y_8)| = d(w(y_1,\dots, y_8), 1) =O(\delta),\]
where $w(t_1, \dots, t_8) := w_*(\tilde w_s(t_1, t_2, t_3, t_4), \tilde w_s(t_5, t_6, t_7, t_8))$. Using the Lipschitz property of $w$, the same is true if $y_1, \dots, y_8 \in N_{\delta}(\phi(Q^2 g))$, of course at the expense of weakening the implicit constant in $O(\delta)$. That is, $N_{\delta}(\phi(Q^2 g))^8 \subset W_{O(\delta)}$, where
\[ W_{\eta} := \{(y_1,\dots, y_8) \in \SO(3)^8 : |w(y_1,\dots, y_8)| \leq \eta\}.\] It therefore suffices to check that $\lim_{\eta \rightarrow 0}\mu^{\otimes 8}(W_{\eta}) = 0$ which, by basic measure theory, is equivalent to the statement that 
\begin{equation}\label{meas-zero} \mu^{\otimes 8}\{ (y_1,\dots, y_8) \in \SO(3)^8 : w(y_1,\dots, y_8) = 1\}= 0.\end{equation}  However, this is so because (see \cite{epstein}) almost all $8$-tuples of elements of $\SO(3)$ generate a free group.
 
\section{Further comments and open questions}\label{sec7}

We have already remarked that it would be interesting to understand more about the structure of approximate homomorphisms $\phi : B \rightarrow \SO(3)$ where $B$ is an approximate group. Does an analogue of Kazhdan's theorem hold for them, that is to say if \eqref{cocycle-small} holds, is there $\tilde \phi : B \rightarrow \SO(3)$ satisfying $\tilde\phi(xy) = \tilde\phi(x) \tilde\phi(y)$ and with $d(\phi(x), \tilde\phi(x)) = O(\delta)$ for all $x$? It might be possible to answer this question using the thesis of Carolino \cite{carolino}, applied to the graph of $\phi$. This would allow for an alternative to the arguments of Section \ref{sec6} by appealing to \cite{breuillard-green}, which says that finite approximate subgroups of $\SO(3)$ are almost abelian. 

The example of Gowers and Long considered in this paper is natural, but has the slightly unsatisfactory property that the operation $\circ$ is not cancellative. It is only weakly cancellative in the sense that for a given $x$ and $z$ there are at most $O(1)$ values of $y$ for which $x \circ y = z$. I have some notes on a potential example which is fully cancellative, so its multiplication table is a latin square. Roughly speaking, it comes from replacing $\SO(3)$ by a compact portion of the Heisenberg group (Jason Long informs me that he and Gowers also considered such examples). I initially thought that the Heisenberg group, being almost abelian, would be much easier to analyse than $\SO(3)$, but this turned out not to be the case. The main reason is that the Heisenberg group does contain approximate subgroups with ``thick'' image. 

The following question, which I cannot currently resolve, came from this line of thinking. Consider the Heisenberg group $\H(\R) = \{(x,y,z) : x, y, z \in \R\}$ with the group operation $\ast$ being $(x_1, y_1, z_1) \ast (x_2, y_2, z_2) = (x_1 + x_2, y_1 + y_2, z_1 + z_2 + x_1 y_2)$.

\begin{question}
Let $K$ be a fixed real parameter and let $N$ be a large integer. Suppose that $B$ is a $K$-approximate group, and that $\phi : B \rightarrow \H(\R)$ is a map with the following properties:
\begin{enumerate}
\item If $\phi$ takes values in $\{(x,y,z) \in \H(\R): x,y \in \frac{1}{N} \Z, |x|, |y|, |z| \leq 10\}$;
\item For every $x, y \in \frac{1}{N}\Z$ with $|x|, |y| \leq 1$ there is some $|z| \leq 10$ such that $(x,y,z) \in \phi(B)$;
\item $\partial\phi$ takes values in $\{(0,0,z) \in \H(\R) : |z| \leq 10/N\}$. 
\end{enumerate}
Must it be the case that $|B|/N^3 \rightarrow \infty$ as $N \rightarrow \infty$?
\end{question}


\begin{thebibliography}{99}
\bibitem{breuillard-green} E.~Breuillard and B.~Green, \emph{Approximate groups III: the unitary case},
Turkish J. Math. \textbf{36} (2012), no. 2, 199--215. 

\bibitem{bgt} E.~Breuillard, B.~Green and T.~Tao, \emph{The structure of approximate groups}, Publ. Math. Inst. Hautes \'Etudes Sci. \textbf{116} (2012), 115--221.

\bibitem{carolino} P.~K.~Carolino, \emph{The Structure of Locally Compact Approximate Groups,} Thesis (Ph.D.), University of California, Los Angeles. 2015. 98 pp.

\bibitem{epstein} D.B.A.~Epstein, \emph{Almost all subgroups of a Lie group are free,}
J. Algebra \textbf{19} (1971), 261--262. 


\bibitem{gowers-1} W.~T.~Gowers, \emph{A new proof of Szemer\'edi's theorem for arithmetic progressions of length four,} Geom. Funct. Anal. \textbf{8} (1998), no. 3, 529--551.

\bibitem{gowers-2} W.~T.~Gowers, \emph{A new proof of Szemer\'edi's theorem,} Geom. Funct. Anal. \textbf{11} (2001), no. 3, 465--588.

\bibitem{gowers-long} W.~T.~Gowers and J.~Long, \emph{Partial associativity and rough approximate groups,} preprint, \texttt{https://arxiv.org/abs/1904.08732}.

\bibitem{gkr} K.~Grove, H.~Karcher, and E.~Ruh, \emph{Jacobi fields and Finsler metrics on compact Lie groups with an application to differentiable pinching problems,}
Math. Ann. \textbf{211} (1974), 7--21. 
 
\bibitem{kazhdan} D.~Kazhdan, \emph{On $\eps$-representations}, Israel J. Math. \textbf{43} (1982), no. 4, 315--323.

\bibitem{sanders} T.~Sanders, \emph{On a nonabelian Balog-Szemer\'edi-type lemma,}
J. Aust. Math. Soc. \textbf{89} (2010), no. 1, 127--132. 

\bibitem{tao} T.~Tao, \emph{Product set estimates for non-commutative groups,} Combinatorica \textbf{28} (2008), no. 5, 547--594. 



\bibitem{turing} A.~M.~Turing, \emph{Finite approximations to Lie groups,} 
Ann. Math. (2) \textbf{39} (1938), no. 1, 105--111. 

\end{thebibliography}
\end{document}